\documentclass[12pt, reqno]{amsart}
\usepackage{amsmath, amsthm, amscd, amsfonts, latexsym, amssymb, graphicx, color}
\usepackage[bookmarksnumbered, colorlinks, plainpages]{hyperref}

\makeatletter \oddsidemargin.9375in \evensidemargin \oddsidemargin
\newtheorem{thm}{Theorem}[section]
\newtheorem{lem}[thm]{Lemma}
\newtheorem{prop}[thm]{Proposition}
\newtheorem{cor}[thm]{Corollary}
\theoremstyle{definition}

\newtheorem{exa}[thm]{Example}

\theoremstyle{remark}
\newtheorem{rem}[thm]{Remark}
\numberwithin{equation}{section}

\begin{document}

\setcounter{page}{1}

\noindent \textbf{{\footnotesize  Journal of Algebra and Related Topics \\
Vol. XX, No XX, (201X), pp XX-XX}}\\[1.00in]
\title[I-prime]{I-prime ideals}
\author[Akray]{Ismael Akray}
\thanks{{\scriptsize
\hskip -0.4 true cm MSC(2010): Primary: 13A15; 
 }}
\begin{abstract}
In this paper, we introduce a new generalization of weakly prime ideals called $I$-prime. Suppose $R$ is a commutative ring with identity and $I$ a fixed ideal of $R$. A proper ideal $P$ of $R$ is $I$-prime if for $a, b \in R$ with $ab \in P-IP$ implies either $a \in P$ or $b \in P$. We give some characterizations of $I$-prime ideals and study some of its properties.  Moreover, we give conditions  under which $I$-prime ideals becomes prime or weakly prime and we construct the view of $I$-prime ideal in decomposite rings.

\end{abstract}

\maketitle

\section{Introduction}
\hskip0.6cm

Throughout this article, $R$ will be a commutative ring with identity. Prime ideals play an essential role in ring theory. A prime ideal $P$ of $R$ is a proper ideal $P$ with the property that for $a,b \in R$, $ab \in P$ implies $a \in P $ or $b \in P$. There are several ways to generalize the notion of a prime ideal. We could either restrict or enlarge where $a$ and/or $b$ lie or restrict or enlarge where $ab$ lies. In this article we will interested in a generalization obtained by restricting where $ab$ lies.  A proper ideal $P$ of $R$ is weakly prime if for $a,b \in R$ with $ab \in P-0$, either $a \in P$ or $b \in P$. Weakly prime ideals were studied in some detail by Anderson and Smith (2003) in \cite{and2003}. Thus any prime ideal is weakly prime. Bhatwadekar and Sharma (2005) in \cite{Sharma} recently deﬁned a proper ideal $I$ of an integral domain $R$ to be almost prime if for $a,b \in R $ with $ab \in I - I^2$, then either $a \in I$ or $b \in I$. This deﬁnition can obviously be made for any commutative ring $R$. Thus a weakly prime ideal is almost prime and any proper idempotent ideal is almost prime. Moreover, an ideal $I$ of $R$ is almost prime if and only if $I/I^2$ is a weakly prime ideal of $R/I^2$. Also almost prime ideals were generalized to $n$-almost prime as follows; a proper ideal $I$ is called $n$-almost prime ideal if for any $a,b \in R $ with $ab \in I - I^n$, then either $a \in I$ or $b \in I$. With weakly prime ideals and almost prime ideals in mind, we make the following deﬁnition. Let $R$ be a commutative ring and $I$ be a fixed ideal of $R$. Then a proper ideal $P$ of $R$ is called $I$-prime ideal if for $a, b \in R$, $ab \in P-IP$, implies $a \in P$ or $b \in P$. So every weakly prime and $n$-almost prime ideal is $I$-prime where $I$ taken to be zero or $P^{n-1}$ respectively. If $I=R$, then every ideal is $I$-prime, so we can assume $I$ to be proper ideal of $R$.
For more details see \cite{Kaplansky}.



\begin{exa}
Consider the ring $\textbf{Z}_{12}$ and take $P=I=<4>$. Then $P$ is $I$-prime which is not prime nor weakly prime.

\end{exa}

\section{Main Results}
\hskip 0.6cm
We begin with the following lemma.

	\begin{lem}
			Let $P$ be a proper ideal of a ring $R$. Then $P$ is an $I$-prime ideal if and only if $P/IP$ is weakly prime in $R/IP$. 
		\end{lem}
		\begin{proof}
		$(\Rightarrow)$	Let $P$ be an $I$-prime in $R$. Let $ a, b \in R$ with $0\neq (a+IP) (b+IP)=ab+IP \in P/IP$. Then $ab \in P-IP$ implies $a \in P$ or $b \in P$, hence $a+IP \in P/IP$ or $b+IP \in P/IP$. So $P/IP$ is weakly prime ideal in $R/IP$. 
		
		$(\Leftarrow)$ Suppose that $P/IP$ is weakly prime in $R/IP$ and take $r, s \in R$ such that $rs \in P-IP$. Then $0\neq rs+IP=(r+IP)(s+IP)\in P/IP$ so $r+IP\in P/IP$ or $s+IP \in P/IP$. Therefore $r \in P$ or $s \in P$. Thus $P$ is an $I$-prime ideal in $R$.
		\end{proof}

	\begin{thm} \label{1}
		(1) Let $I\subseteq J$. If $P$ is $I$-prime ideal of a ring $R$, then it is $J$-prime.
		\\(2) Let $R$ be commutative ring and $P$ an $I$-prime ideal that is not prime, then $P^{2}\subseteq IP$. Thus, an $I$-prime ideal $P$ with $P^{2}\nsubseteq IP$ is prime.
	\end{thm}
	
	\begin{proof}
	(1) The proof come from the fact that if $I\subseteq J$, then $P-JP \subseteq P-IP$.
		(2) Suppose that $P^{2}\nsubseteq IP$, we show that $P$ is prime. Let $ab\in P$ for $a,b\in R$. If $ab\notin IP$, then $P$ $I$-prime gives $a\in P$ or $b\in P$. So assume that $ab\in IP$. First, suppose that $aP\nsubseteq IP$; say $ax\notin IP$ for some $x \in P$. Then $a(x+b)\in P-IP$. So $a\in P$ or $x+b\in P$ and hence $a\in P$ or $b\in P$. So we can assume that $aP\subseteq IP$ in similar way we can assume that $bP\subseteq IP$. Since $P^{2}\nsubseteq IP$, there exist $y,z\in P$ with $yz\notin IP$. Then $(a+y)(b+z)\in P-IP$. So $P$ $I$-prime gives $a+y\in P$ or $b+z\in P$; Hence $a\in P$ or $b\in P$. Therefore $P$ is prime.
	\end{proof}

In the following we give a counter example on the converse of part (1) of Theorem \ref{1}.
\begin{exa}
In the ring $\textbf{Z}_{12}[x]$, put $I=0$, $J=<4>$ and $P=<4x>$. Then $P-IP=<4x>-0$ and $P-JP=<4x>-<4><4x>= \phi $. Hence $P$ is $J$-prime but not $I$-prime.

\end{exa}	
	
	\begin{cor}
		Let $P$ be an $I$-prime ideal of a ring $R$ with $IP\subseteq P^{3}$. Then $P$ is $\cap _{i=1}^{\infty} P^{i}$-prime. 
	\end{cor}
	\begin{proof}
		If $P$ is prime, then $P$ is $\cap _{i=1}^{\infty} P^{i}$-prime. Assume that $P$ is not prime. By Theorem \ref{1} $P^{2}\subseteq IP\subseteq P^{3}$. Thus $IP=P^{n}$ for each $n\geq 2$. So $\cap _{i=1}^{\infty} P^{i}=P \cap P^2=P^2$ and $(\cap _{i=1}^{\infty} P^{i}) P=P^2 P=P^3=IP$. Hence $P$ $I$-prime implies $P$ is $\cap _{i=1}^{\infty} P^{i}$-prime. 
	\end{proof}

	\begin{rem}
		Let $P$ be $I$-prime ideal. Then $P\subseteq \sqrt{IP}$ or $\sqrt{IP}\subseteq P$. If $P \subsetneqq \sqrt{IP}$, then $P$ is not prime since otherwise $IP\subseteq P$ implies $\sqrt{IP}\subseteq \sqrt{P}=P$. While if $\sqrt{IP}\subsetneqq P$, then $P$ is prime. 
	\end{rem}
	
	\begin{cor}
		Let $P$ be $I$-prime ideal of a ring $R$ which is not prime. Then $\sqrt{P} = \sqrt{IP}$
			\end{cor}
			
			\begin{proof}
			By Theorem \ref{1}, $P^2 \subseteq IP$ and hence $\sqrt{P} =\sqrt{P^2} \subseteq \sqrt{IP}$. The other containment always holds.
		
			\end{proof}

		Now we give a way to construct $I$-prime ideals $P$ when $\cap _{i=1}^{\infty} P^{i}\subseteq IP \subseteq P^{3}$. 
	
	\begin{rem}
		Assume that $P$ is $I$-prime, but not prime. Then by Theorem \ref{1}, if $IP\subseteq P^{2}$, then $P^{2}=IP$. In particular, if $P$ is weakly prime ($0$-prime) but not prime, then $P^{2}=0$. Suppose that $IP\subseteq P^{3}$. Then $P^{2}\subseteq IP\subseteq P^{3}$; So $P^{2}=P^{3}$ and thus $P^{2}$ is idempotent. 
	\end{rem}
	
	\begin{thm} \label{4}
		(1) Let $R,S$ be commutative rings and $P$ $0$-prime ideal of $R$. Then $P\times S$ is $I$-prime ideal of $R \times S$ for each ideal $I$ of $R \times S$ with $\cap _{i=1} ^{\infty} (P\times S)^i\subseteq I(P\times S)\subseteq P\times S$. 
		\\(2) Let $P$ be finitely generated proper ideal of commutative ring $R$. Assume $P$ is $I$-prime where $IP\subseteq P^{3}$. Then either $P$ is $0$-prime or $P^{2}\neq 0$ is idempotent and $R$ decomposes as $T \times S$ where $S=P^2$ and $P=J \times S$ where $J$ is $0$-prime. Thus $P$ is $I$-prime for $\cap _{i=1} ^{\infty} P^i \subseteq IP \subseteq P$.
	\end{thm}
	\begin{proof}
		(1) Let $R$ and $S$ be commutative rings and $P$ be weakly prime ($0$-prime) ideal of $R$. Then $P\times S$ need not be a $0$-prime ideal of $R\times S$; In fact, $P\times S$ is $0$-prime if and only if $P\times S$ (or equivalently $P$) is prime [see Anderson 2003]. However, $P\times S$ is $I$-prime for each $I$ with $\cap _{i=1}^{\infty} (P\times S)^{i}\subseteq I(P\times S)$. If $P$ is prime , then $P\times S$ is prime ideal and thus is $I$-prime for all $I$. Assume that $P$ is not prime. Then  $P^{2}=0$ and $(P\times S)^{2}=0\times S$. Hence $\cap _{i=1}^{\infty} (P\times S)^{i}= \cap_{i=1}^{\infty} P^{i}\times S=0\times S $. Thus $P\times S -\cap _{i=1}^{\infty} (P\times S)^{i} = P\times S - 0\times S=(P-0)\times S$. Since $P$ is weakly prime, $P\times S$ is $\cap _{i=1}^{\infty} (P\times S)^{i}$-prime and as $\cap _{i=1}^{\infty} (P\times S)^{i}\subseteq I(P\times S)$, $P\times S$ is $I$-prime. 
		
		(2) If $P$ is prime, then $P$ is $0$-prime. So we can assume that $P$ is not prime. Then $P^{2} \subseteq IP$ and hence $P^{2} \subseteq IP \subseteq P^{3}$. So $P^{2}=P^{3}$. Hence $P^{2}$ is idempotent. Since $P^{2}$ is finitely generated, $P^{2}=<e>$ for some idempotent $e \in R$. Suppose $P^{2}=0$. Then $IP \subseteq P^3 =0$. So $IP=0$ and hence $P$ is $0$-prime. So assume $P^{2} \neq 0$. Put $S=P^{2}=Re$ and $T = R(1-e)$, so $R$ decomposesas $T \times S$ where $S=P^{2}$. Let $J=P(1-e)$, so $P=J \times S$ where $J^2 =(P(1-e))^2=P^2(1-e)^2=<e> (1-e)=0$. We show that $J$ is $0$-prime. Let $ab \in J^2-0$, so $(a,1)(b,1)=(ab,1) \in J \times S-(J \times S)^2= J \times S - 0\times S \subseteq P - IP$ since $IP\subseteq P^{3}$ implies $IP \subseteq P^3 =(J \times S)^3 = 0 \times S$. Hence $(a,1) \in P$ or $(b,1) \in P$ so $ a \in J$ or $b \in J$. Therefore $J$ is weakly prime.
	\end{proof}
	
	\begin{cor}
		Let $R$ be an indecomposable commutative ring and $P$ a finitely generated $I$-prime ideal of $R$, where $IP\subseteq P^{3}$. Then $P$ is weakly prime. 
	\end{cor}

	\begin{cor}
		Let $R$ be a Noetherian integral domain. A proper ideal $P$ of $R$ is prime if and only if $P$ is $P^{2}$-prime. 
	\end{cor}

	\begin{thm}
		Let $a$ be a non-unit element in $R$.
		\\ (1) Let $(0:a)\subseteq (a)$. Then $(a)$ is $I$-prime for $I(a)\subseteq (a)^{2}$ if and only if $(a)$ is prime.
		\\ Let $(R,m)$ be quasi-local ring. Then \\
		(2) $(a)$ is $I$-prime for $I(a)\subseteq (a)^{3}$ if and only if $(a)$ is $0$-prime.
		\\(3) $(a)$ is $m$-prime if and only if  $a$ is irreducible.
	\end{thm}
	\begin{proof}
		(1) Suppose that (a) is $I$-prime and $bc\in (a)$. If $bc\notin I(a)$, then $b\in (a)$ or $c\in (a)$. So suppose that $bc\in I(a)$. Now $(b+a)c\in (a)$. If $(b+a)c\notin I(a)$, then $b+a\in(a)$ or $c\in (a)$ and hence $b\in (a)$ or $c\in(a)$. So assume that $(b+a)c\in I(a)$. Then $ac\in I(a)$ and hence $ac\in (a)^{2}$. So $ac=za^{2}$ and hence $c-za\in (0:a)$. Thus $c\in (0:a)+(a)=(a)$. The converse part is trivial since every prime ideal is $I$-prime.
		\\ (2) If $(a)$ is $0$-prime, then $(a)$ is $I$-prime for each $I$ with $I(a)\subseteq (a)^{3}$. Conversely, let $(a)$ be $I$-prime for $I(a)\subseteq (a)^{3}$. Since a quasi local ring has no nontrivial idempotents, $(a)$ is $0$-prime by Theorem \ref{4} part(2). 
		\\(3) If $a$ is irreducible means that $a=xy$ implies that $x\in(a)$ or $y\in (a)$ and $(a)$ is $m$-prime means that $xy\in (a)-m(a)$ which implies that $x\in (a)$ or $y\in (a)$. But $xy \in (a)-m(a)$ if and only if $xy=za$ for some unit $z\in R$ if and only if $a=z^{-1}xy$ for some unit $z^{-1}\in R$. Thus $(a)$ is $m$-prime if and only if $a=xy$ implies $x\in (a)$ or $y\in (a)$.
	\end{proof}

	We now give some characterizations of $I$-prime ideals.
	
	\begin{thm} \label{13}
		Let $P$ be a proper ideal of $R$. Then the following are equivalent:
		\\ (1) $P$ is $I$-prime.
		\\(2) For $x\in R-P$, $(P:x)=P\cup (IP:x)$
		\\ (3) For $x\in R-P$, $(P:x)=P$ or $(P:x)=(IP:x)$
		\\(4) For ideals $J$ and $K$ of $R$, $JK\subseteq P$ and $JK\nsubseteq IP$ imply $J\subseteq P$ or $K\subseteq P$. 
	\end{thm}
	\begin{proof}
		(1) $ \Rightarrow$ (2) Suppose $r\in R-P$. Let $s\in (P:r)$, so $rs\in P$. If $rs\in P-IP$, then $s\in P$. If $rs\in IP$, then $s\in (IP:r)$, So $(P:r)\subseteq P\cup (IP:r)$. The other containment always holds. 
		\\(2) $ \Rightarrow $ (3) Note that if an ideal is a union of two ideals, then it is equal to one of them.
		\\(3) $ \Rightarrow $ (4) Let $J$ and $K$ be two ideals of $R$ with $JK\subseteq P$. Assume that $J\nsubseteq P$ and $K\nsubseteq P$. We claim that $JK\subseteq IP$. Suppose $r\in J$. First, Let $r \notin P$. Then $rK\subseteq P$ gives $K\subseteq (P:r)$. Now $K\nsubseteq P$, so $(P:r)=(IP:r)$. Thus $rK\subseteq IP$. Next, let $r\in J\cap P$. Choose $s\in J-P$. Then $r+s\in J-P$. So by the first case $sK\subseteq IP$ and so $(r+s)K\subseteq IP$. Let $t\in K$. Then $rt=(r+s)t-st\in IP$. So $rK\subseteq IP$. Hence $JK\subseteq IP$. 
		\\(4) $ \Rightarrow $ (1) Let $rs\in P-IP$. Then $(r)(s)\subseteq P$ but $(r)(s)\nsubseteq IP$. So $(r)\subseteq P$ or $(s)\subseteq P$; i-e. $r\in P$ or $s\in P$.
	\end{proof}
	
	\begin{cor}
		Suppose $P$ is $I$-prime ideal that is not prime. Then $P\sqrt{IP}\subseteq IP$.
	\end{cor}
	\begin{proof}
		Let $r\in \sqrt{IP}$. If $r\in P$, then $rP\subseteq P^{2}\subseteq IP$ by Theorem \ref{1}. So assume that $r\notin P$ by Theorem \ref{13}, $(P:r)=P$ or $(P:r)=(IP:r)$. As $P\subseteq (P:r)$, the last gives $rP\subseteq IP$. So assume that $(P:r)=P$. Let $r^{n}\in IP$, but $r^{n-1}\notin IP$. Then $r^{n}\in P$, so $r^{n-1}\in (P:r)=P$. Thus $r^{n-1}\in P-IP$, so $r\in P$ a contradiction. 
	\end{proof}

	It is known that if $S$ is a multiplicatively closed subset of a commutative ring $R$ and $P$ as a prime ideal of $R$ with $P\cap S= \phi$, then $S^{-1}P$ is a prime ideal of $S^{-1}R$ and $S^{-1}P\cap R=P$. The first result extended to weakly prime ideals in [2, Proposition 13] and to almost prime ideals in [5, Lemma 2.13]. Fix an ideal $I$ of $R$ we prove that if $P$ is $I$-prime with $P\cap S= \phi$, then $S^{-1}P$ is $S^{-1}I$-prime. Note that for an ideal $J$ of $R$ with $J \subseteq P$, $I(P/J)=(IP+J)/J$. If $P$ is prime (respectively, weakly prime, $n$-almost prime), then so is $P/J$. We generalize this result to $I$-prime ideals in the following proposition.
	
	\begin{prop}
		Let $R$ be a ring and $I$ be an ideal of $R$. Let $P$ be $I$-prime ideal of $R$. Then the following are true.
		\\(1) If $J$ is an ideal of $R$ with $J\subseteq P$, then $P/J$ is $I$-prime ideal of $R/J$. 
		\\(2) Assume that $S$ is multiplicatively closed subset of $R$ with $P\cap S=\phi$. Then $S^{-1}P$ is a $S^{-1}I$-prime ideal of $S^{-1}R$. If $S^{-1}P\neq S^{-1}(IP)$, then $S^{-1}P \cap R=P$.  
	\end{prop}
	\begin{proof}
		(1) Let $x,y\in R$ with $\bar{x} \bar{y}\in P/J-I(P/J)=P/J-(IP+J)/J$. Thus $xy\in P-(IP+J)$. Hence $xy\in P-IP$, so $x\in P$ or $y\in P$. Therefore $\bar{x}\in P/J$ or $\bar{y}\in P/J$; so $P/J$ is $I$-prime. 
		\\(2) Let $\frac{a}{s} \frac{b}{t}\in S^{-1}P-S^{-1}I.S^{-1}P\subseteq S^{-1}P-S^{-1}(IP)=S^{-1}(P-IP)$. So $abk\in P-IP$ for some $k\in S$. Thus $P$ $I$-prime gives $a\in P$ or $bk\in P$. So $a/s\in S^{-1}P$ or $b/t\in S^{-1}P$. Hence $S^{-1}P$ is $S^{-1}I$-prime. Let $a\in S^{-1}P\cap R$, so there exists $u\in S$ with $au\in P$. If $au\notin IP$, then $au\in P-IP$, so $a\in P$. If $au\in IP$, then $a\in S^{-1}(IP)\cap R$. So $S^{-1}P \cap R \subseteq P \cup (S^{-1}(IP) \cap R)$. Hence $S^{-1}P \cap R=P$ or $S^{-1}P \cap R= S^{-1}(IP) \cap R$. But the second case gives $S^{-1}P =S^{-1}(IP)$. 
	\end{proof}

	Let $R_{1}$ and $R_{2}$ be two rings. It is known that the prime ideals of $R_{1}\times R_{2}$ have the form $P\times R_{2}$ or $R_{1}\times Q$, where $P$ is a prime ideal of $R_{1}$ and $Q$ is a prime ideal of $R_{2}$. We next, generalize this result to $I$-primes.

	\begin{thm}
		For $i=1,2$ let $R_{i}$ be ring and $I_{i}$ ideal of $R_{i}$. Let $I=I_{1}\times I_{2}$. Then the $I$-prime ideals of $R_{1}\times R_{2}$ have exactly one of the following three types:
		\\(1) $P_{1}\times P_{2}$, where $P_{i}$ is a proper ideal of $R_{i}$ with $I_{i}P_{i}=P_{i}$.
		\\(2) $P_{1}\times R_{2}$ where $P_{1}$ is an $I_{1}$-prime of $R_{1}$ and $I_{2}R_{2}=R_{2}$. 
		\\(3) $R_{1}\times P_{2}$, where $P_{2}$ is an $I_{2}$-prime of $R_{2}$ and $I_{1}R_{1}=R_{1}$.
	\end{thm}
	\begin{proof}
		We first prove that an ideal of $R_{1}\times R_{2}$ having one of these three types is $I$-prime. The first type is clear since $P_{1}\times P_{2}-I(P_{1}\times P_{2})=P_{1}\times P_{2}-(I_{1}P_{1}\times I_{2}P_{2})=\phi$. Suppose that $P_{1}$ is $I_{1}$-prime and $I_{2}R_{2}=R_{2}$. Let $(a,b)(x,y)\in P_{1}\times R_{2}-I_{1}P_{1}\times I_{2}R_{2}=P_{1}\times R_{2}-I_{1}P_{1}\times R_{2}=(P_{1}-I_{1}P_{1})\times R_{2}$. Then $ax\in P_{1}-I_{1}P_{1}$ implies that $a\in P_{1}$ or $x\in P_{1}$, so $(a,b)\in P_{1}\times R_{2}$ or $(x,y)\in P_{1}\times R_{2}$. Hence $P_{1}\times R_{2}$ is $I$-prime. Similarly we can prove the last case.
	
	Next, let $P_{1}\times P_{2}$ be $I$-prime and $ab\in P_{1}-I_{1}P_{1}$. Then $(a,0)(b,0)=(ab,0)\in P_{1}\times P_{2}-I(P_{1}\times P_{2})$, so $(a,0)\in P_{1}\times P_{2}$ or $(b,0)\in P_{1}\times P_{2}$, i-e, $a\in P_{1}$ or $b\in P_{1}$. Hence $P_{1}$ is $I_{1}$-prime. Likewise, $P_{2}$ is $I_{2}$-prime.
	
	Assume that $P_{1}\times P_{2}\neq I_{1}P_{1}\times I_{2}P_{2}$. Say $P_{1}\neq I_{1}P_{1}$. Let $x\in P_{1}-I_{1}P_{1}$ and $y\in P_{2}$. Then $(x,1)(1,y)=(x,y)\in P_{1}\times P_{2}$. So $(x,1)\in P_{1}\times P_{2}$ or $(1,y)\in P_{1}\times P_{2}$. Thus $P_{2}= R_{2}$ or $P_{1}=R_{1}$. Assume that $P_{2}=R_{2}$. So $P_{1}\times R_{2}$ is $I$-prime, where $P_{1}$ is $I_{1}$-prime.
	\end{proof}

\vskip 0.4 true cm

\begin{center}{\textbf{Acknowledgments}}
\end{center}
The author is deeply grateful to referee for his careful reading of the manuscript. Whose comments made the paper more readable.\\ \\
\vskip 0.4 true cm

\bibliographystyle{amsplain}

\bigskip
\bigskip

{\footnotesize {\bf Ismael Akray}\; \\ {Department of
Mathematics}, {University
of Soran} {Kurdistan region-Erbil, Iraq.}\\
{\tt Email: ismael.akray@soran.edu.iq  \hspace{.5cm}  ismaeelhmd@yahoo.com}\\



\vskip.5cm

\end{document}